\documentclass[a4paper]{amsart}
\usepackage{color}
\usepackage{enumerate}
\usepackage{graphics}
\usepackage{mathrsfs}
\usepackage{amssymb}

\newtheorem{theorem}{Theorem}
\newtheorem*{theorem*}{Theorem}
\newtheorem*{problem*}{Problem}
\newtheorem{lemma}[theorem]{Lemma}
\newtheorem{proposition}[theorem]{Proposition}
\newtheorem{corollary}[theorem]{Corollary}

\theoremstyle{remark}

\newtheorem{remark}[theorem]{Remark}

\newcommand{\norm}[1]{\Vert#1\Vert}
\newcommand{\bignorm}[1]{\bigl\Vert#1\bigr\Vert}
\newcommand{\Bignorm}[1]{\Bigl\Vert#1\Bigr\Vert}

 \newcommand{\Cdb}{\mbox{$\mathbb{C}$}}
 \newcommand{\Tdb}{\mbox{$\mathbb{T}$}}

 \newcommand{\Rdb}{\mbox{$\mathbb{R}$}}

\begin{document}

\title[Peller problem]{Peller's problem concerning
Koplienko-Neidhardt trace formulae : the unitary case}

\author[C. Coine]{Clement Coine}
\email{clement.coine@univ-fcomte.fr}
\author[C. Le Merdy]{Christian Le Merdy}
\email{clemerdy@univ-fcomte.fr}
\author[D. Potapov]{Denis Potapov}
\email{d.potapov@unsw.edu.au}
\author[F. Sukochev]{Fedor Sukochev}
\email{f.sukochev@unsw.edu.au}
\author[A. Tomskova]{Anna Tomskova}
\email{a.tomskova@unsw.edu.au}

\address{D.P, F.S, A.T: School of Mathematics \& Statistics, University of NSW,
Kensington NSW 2052, AUSTRALIA}
\address{C.C, C.L: Laboratoire de Math\'{e}matiques, Universit\'{e} de Franche-comt\'{e},
25030 Besan\c{c}on Cedex, FRANCE}

\maketitle

\begin{abstract} We prove the existence of a $C^2$-function $f\colon\Tdb\to\Cdb\,$
defined on the unit circle, a unitary operator $U$ and a self-adjoint operator
$Z$ in the Hilbert-Schmidt class $\mathcal S^2$, such that
$$
f(e^{iZ}U)-f(U) -\frac{d}{dt}\bigl(f(e^{itZ}U)\bigr)_{\vert t=0}\,\notin\, \mathcal S^1,
$$
the space of trace class operators. This resolves a problem of Peller concerning
the validity of the Koplienko-Neidhardt trace formula for unitaries.
\end{abstract}


\bibliographystyle{short}

\section{Introduction}

Let $\mathcal H$ be a Hilbert space and let $B(\mathcal H)$ be the algebra of all
bounded operators on $\mathcal H$ equipped with the standard trace ${\rm Tr}$. For any $1\leq p\leq \infty$, let
$\mathcal S^p(\mathcal H)$ denote the Schatten $p$-class over $\mathcal H$. Let 
$\Tdb=\{z\in\Cdb\, :\, \vert z\vert=1\}$ denote the unit circle of the
complex plane. Let $f$ be a function on $\Tdb $, admitting a decomposition 
 $f(z)=\sum _{n=-\infty}^{\infty} c_n z^n$, $z\in \Tdb$ with 
$\sum _{n=-\infty}^{\infty} |nc_n|<\infty $. Let $U\in B(\mathcal H)$ 
be a unitary operator and let $Z\in \mathcal S^1(\mathcal H)$ be a self-adjoint operator. 
In 1962, M. G. Krein proved a result (see \cite[Theorem 2]{Krein}) implying that
there exists a unique function $\eta\in L^1(\mathbb T)$ (not depending on $f$) such that
\begin{equation}\label{KreinsFormula}
{\rm Tr}\bigl(f(e^{iZ}U)-f(U)\bigr)=\int_{\mathbb T} f'(z) \eta(z) dz.
\end{equation}
The function $\eta$ above is called Lifshits-Krein spectral shift function, 
it plays an important role in scattering theory, where it appears in the
formula of the determinant of scattering matrix (for detailed discussion we refer 
to \cite{BY} and references therein, see also self-adjoint version of formula 
\eqref{KreinsFormula} in \cite{Krein1}).

Observe that the right-hand side of $\eqref{KreinsFormula}$ 
makes sense for every Lipschitz function $f$.
In 1964, M. G. Krein \cite{KreinPerturbation1964} discussing a self-adjoint version 
of formula \eqref{KreinsFormula} (introduced in 1953, see \cite[Theorem 4]{Krein1}) 
conjectured that the left-hand side of $\eqref{KreinsFormula}$ also makes sense for 
every Lipschitz function $f$.

The best result to date concerning the description of the class of functions for which 
the left-hand side of $\eqref{KreinsFormula}$  makes sense
is due to V. Peller in \cite{Peller1985}, who established that for 
 $f\in  B^1_{\infty 1}$ (for definition of the Besov
classes see \cite{Peller1985} and references therein).
However, there is an example of a continuously differentiable  function $f$, 
a unitary operator $U$ and a
self-adjoint operator  $Z\in\mathcal S^1(\mathcal H)$ such that 
$$
f(e^{iZ}U)-f(U)\notin \mathcal S^1(\mathcal H).
$$
Such an example can be found in \cite{Peller1985} (see also additional 
discussion and references in \cite{Farforovskaya1972}, 
and \cite{PS-Lipschitz}, \cite{ACS}, \cite{CPSZ1}, \cite{CPSZ2}).

Let now $f\in C^2(\Tdb)$, let $U\in B(\mathcal H)$
be a unitary operator and let $Z\in \mathcal S^2(\mathcal H)$ be a self-adjoint operator. 
Then the difference operator $f(e^{iZ}U)-f(U)$ belongs to $\mathcal S^2(\mathcal H)$
and the function $t\mapsto f(e^{itZ}U) -f(U)$ from $\Rdb$ into $\mathcal S^2(\mathcal H)$
is differentiable, see e.g. \cite[(2.7)]{Peller2005}. Let $\frac{d}{dt}\bigl(f(e^{itZ}U)\bigr)_{\vert t=0}$
denote its derivative at $t=0$. In \cite[Problem 1]{Peller2005}, in connection with the validity 
of the so-called Koplienko-Neidhardt trace formula, V. V. Peller asked whether the 
operator 
\begin{equation}\label{second}
f(e^{iZ}U)-f(U) -\frac{d}{dt}\bigl(f(e^{itZ}U)\bigr)_{\vert t=0}
\end{equation}
necessarily belongs to $\mathcal S^1(\mathcal H)$ under these assumptions. 
He proved that this holds true 
whenever $f$ belongs to the Besov class $B^{2}_{\infty 1}$ and derived a
Koplienko-Neidhardt trace formula in this case. The main purpose of this
paper is to devise a counter-example which shows that Peller's question has a 
negative answer, see Theorem \ref{main} below.

In the preceding paper \cite{CLPST1} we proved the following result: there exists a
$C^2$-function $f\colon\Rdb\to\Rdb\,$ with a bounded second derivative, 
a self-adjoint (unbounded) operator 
$A$ on $\mathcal H$ and a self-adjoint operator $B$ in $\mathcal S^2(\mathcal H)$
such that 
$$
f(A+B) - f(A) - \frac{d}{dt}\bigl(f(A+tB)\bigr)_{\vert t=0}\,\notin\, \mathcal S^1(\mathcal H).
$$
This answered in negative another question raised by V. V. Peller in \cite{Peller2005}.  
It should be noted that the first step of the proof of Theorem \ref{main} follows the 
proof of the main result of \cite{CLPST1}, in the sense 
that we apply a formula obtained in \cite[Theorem 6]{CLPST1} (restated below as Theorem \ref{key}). 
However the key ideas of our approach here are completely different from those in \cite{CLPST1}. Indeed,  
to construct our example we consider bounded (unitary) operators only, and therefore we have 
to work with functions whose derivatives have singular points belonging to 
$\mathbb T$, whereas in \cite{CLPST1} such points were based at infinity. Our analysis 
here is partly based on results from \cite{AdPS}, where some fine estimates for 
operator-functions of such type were obtained.
 
In Section 2 we give some background and preliminary results on 
bilinear Schur products and multiple 
operator integrals in the finite dimensional setting. 
In Section 3 we establish a new formula relating the operator
(\ref{second}) to the actions of appropriate multiple 
operator integrals. Section 4 consists of various finite dimensional estimates 
concerning multiple operator integrals. The main result is established in Section 5.

We end this introduction with a few notation.
Throughout we let $\sigma(A)$ denote the spectrum of an operator $A$ and we let 
$\norm{\ }_p$ denote the norm on the Schatten space $\mathcal S^p(\mathcal H)$.
For any integer $n\geq 1$, we let $\ell^2_n$ be the space $\Cdb^n$ equipped
with its standard Hilbertian structure and we let $M_n$ be the space of all $n\times
n$ matrices with entries in $\Cdb$. Further for any $1\leq p\leq\infty$, we let
$\mathcal S^{p}_n$ denote this matrix space equipped 
with the Schatten $p$-norm $\|\cdot\|_{p}$.

\section{Multiple operator integrals in the finite dimensional case}

For any double-indexed family
$M =\{m_{ij}\}_{i,j=1}^n$, we let $L_M\colon M_n\to M_n$
denote the linear Schur multiplier defined by
$$
L_M(X) = [m_{ij}x_{ij}],\qquad X=[x_{ij}]\in M_n.
$$
Likewise for any triple-indexed family $M =\{m_{ikj}\}_{i,j,k=1}^n$,
we let $B_M\colon M_n\times M_n\to M_n$
denote the bilinear Schur multiplier defined by
$$
B_M(X,Y) = \biggl[\sum_{k=1}^n m_{ikj}x_{ik}y_{kj}\biggr],\qquad X=[x_{ij}], \ Y=[y_{ij}] \in M_n.
$$
The following result from \cite{CLPST1} will provide a key estimate in the 
resolution of Peller's problem.

\begin{theorem}\cite[Theorem 6]{CLPST1}\label{key}
Consider a triple-indexed family
$M=\{m_{ikj}\}_{i,j,k=1}^n$ and for any $k=1,\ldots,n$,
set $M(k)=\{m_{ikj}\}_{i,j=1}^n$. Then we have
$$
\bignorm{B_M \colon \mathcal S^2_n\times \mathcal S^2_n\to \mathcal S^1_n}\,=\,
\sup_{1\le k\le n}\bignorm{L_{M(k)}\colon \mathcal S^\infty_n\to \mathcal S^\infty_n}.
$$
\end{theorem}

We now present the finite dimensional versions of double operator integrals (resp. 
triple operator integrals) associated to a pair (resp. a triple) of normal
operators. We follow \cite[Subsections 3.1 and 3.2]{CLPST1}. In the latter paper,
we considered self-adjoint operators only, however the extension to the normal case
is straightforward.

Let $U_0,U_1\in B(\ell^2_n)$ be normal operators. For 
$j=0,1$, consider an orthonormal basis $\{\xi_i^{(j)}\}_{i=1}^n$
of eigenvectors for $U_j$, let 
$\{z_i^{(j)}\}_{i=1}^n$ be the associated
$n$-tuple of eigenvalues, that is, $U_j(\xi_i^{(j)})=z_i^{(j)}\xi_i^{(j)}$, and let
$P_i^{(j)}$ denote the orthogonal projection onto 
the linear span of $\xi_i^{(j)}$. Then $U_0$ and $U_1$ have spectral decompositions
$$
U_0 = \sum_{i=1}^n z_i^{(0)} P_i^{(0)}
\qquad\hbox{and}\qquad
U_1 = \sum_{k=1}^n z_k^{(1)} P_k^{(1)}.
$$
For any function $\phi\colon\Cdb^2\to \Cdb$, we let
$T_\phi^{U_0,U_1}:B(\ell^2_n)\to B(\ell^2_n)$ 
be the linear operator defined by
$$
T_\phi^{U_0,U_1}(X)=\sum_{i,k=1}^{n} \phi(z_i^{(0)},
z_k^{(1)})P_i^{(0)} X P_k^{(1)}, \
\ X\in B(\ell^2_n).
$$

Next let $U_2\in B(\ell^2_n)$ be a third normal operators. Again consider
a spectral decomposition 
$$
U_2 = \sum_{j=1}^n z_j^{(2)} P_j^{(2)},
$$ 
that is, $P_1^{(2)},\ldots, P_n^{(2)}$
are pairwise orthogonal rank one projections and
$z_1^{(2)},\ldots, z_n^{(2)}$ are eigenvalues of $U_2$. 

For any $\psi\colon\Cdb^3\to \Cdb$, we let
$T_\psi^{U_0,U_1,U_2}:B(\ell^2_n)\times B(\ell^2_n)\to B(\ell^2_n)$ 
be the bilinear operator defined by
\begin{equation}\label{bilinearMOI}
T_\psi^{U_0,U_1,U_2}(X,Y)=\sum_{i,j,k=1}^{n} \psi(z_i^{(0)},
z_k^{(1)},z_j^{(2)})P_i^{(0)} X P_k^{(1)}YP_j^{(2)}, \
\ X,Y\in B(\ell^2_n).
\end{equation}

The following results relate the norms 
of the above operators to the norms of certain Schur mutlipliers.
The (easy) proofs of these equalities are explained in \cite{CLPST1}.

\begin{lemma}\label{corres} 
Let $U_0,U_1, U_2\in B(\ell^2_n)$ be normal operators.
\begin{itemize} 
\item[(a)]
For any function $\phi\colon\Cdb^2\to\Cdb$, consider the double-indexed 
family $M_\phi = \bigl\{\phi(z_i^{(0)}, z_k^{(1)})\bigr\}_{i,k=1}^n$. Then 
$$
\bignorm{T_\phi^{U_0,U_1}\colon \mathcal S^\infty_n\to \mathcal S^\infty_n}
\,=\,
\bignorm{L_{M_\phi}\colon \mathcal S^\infty_n\to \mathcal S^\infty_n}.
$$
\item[(b)]
For any function $\psi\colon\Cdb^3\to\Cdb$, consider the triple-indexed 
family $M_\psi = \bigl\{\psi(z_i^{(0)}, z_k^{(1)}, z_j^{(2)})\bigr\}_{i,j,k=1}^n$. Then 
$$
\bignorm{T_\psi^{U_0,U_1,U_2}\colon \mathcal S^2_n\times
\mathcal S^2_n\to \mathcal S^1_n}\,=\,
\bignorm{B_{M_{\psi}} 
\colon \mathcal S^2_n\times
\mathcal S^2_n\to \mathcal S^1_n}.
$$
\end{itemize}
\end{lemma}

The definition (\ref{bilinearMOI}) only depends on the value
of $\psi$ on the product of the spectra of the operators $U_0,U_1,U_2$.
Hence in the definition of $T_\psi^{U_0,U_1, U_2}$,
the function $\psi$ could be defined only on a subset of
$\Cdb^3$ containing the product of these spectra. 
A similar comment applies
to the definition of $T_\phi^{U_0,U_1}$. In the sequel, 
the normal operators $U_j$ will be unitaries and we will
deal with functions $\psi$ (resp. $\phi$) defined on $\Tdb^3$
(resp. on $\Tdb^2$). 

We will need the following approximation lemma.

\begin{lemma}\label{lem_contin}
Let $U_0,U_1,U_2\in B(\ell^2_n)$ be unitary operators 
and let $(F_m)_m$ be a sequence of unitaries such that $F_m\to U_0$ 
 in the uniform operator topology as 
$m\to\infty$. Let $\psi\in C(\Tdb^3)$. Then 
$$
T^{F_m,U_1,U_2}_{\psi}\longrightarrow
T^{U_0,U_1,U_2}_{\psi} \quad\hbox{as}\ m\to\infty.
$$
\end{lemma}

\begin{proof} 
Let $F\in B(\ell^2_n)$ be any unitary operator. 
Consider a spectral decomposition $F = \sum_{i=1}^n z_i P_i$.
Let $X,Y\in B(\ell^2_n)$. 
According to (\ref{bilinearMOI}), we have
\begin{align*}
T_\psi^{F,U_1,U_2}(X,Y) & =\sum_{i,j,k=1}^{n} \psi(z_i,
z_k^{(1)},z_j^{(2)})P_i X P_k^{(1)}YP_j^{(2)}\\
& = \sum_{j,k=1}^{n} \Bigl(\sum_{i=1}^n \psi(z_i,
z_k^{(1)},z_j^{(2)})P_i\Bigr) X P_k^{(1)}YP_j^{(2)}\\
& = \sum_{j,k=1}^{n} \psi(F, z_k^{(1)},z_j^{(2)})X P_k^{(1)}YP_j^{(2)},
\end{align*}
where $\psi(F, z_k^{(1)},z_j^{(2)})$ is the operator obtained by applying
the continuous functional calculus of $F$ to 
$\psi(\cdotp,z_k^{(1)},z_j^{(2)})$. 

For any $\varphi\in C(\Tdb)$, the mapping $F\mapsto \varphi(F)$ is continuous from
the set ot unitaries of $B(\ell^2_n)$ into $B(\ell^2_n)$.
Hence for any $j,k=1,\ldots, n$,
$$
\psi(F_m, z_k^{(1)},z_j^{(2)})\longrightarrow
\psi(U_0, z_k^{(1)},z_j^{(2)})\quad\hbox{as}\ m\to\infty.
$$
From the above computation we deduce that for any $X,Y\in B(\ell^2_n)$,
$$
T^{F_m,U_1,U_2}_{\psi}(X,Y)\longrightarrow
T^{U_0,U_1,U_2}_{\psi}(X,Y) \quad\hbox{as}\ m\to\infty.
$$
Since $T^{F_m,U_1,U_2}_{\psi}$ and 
$T^{U_0,U_1,U_2}_{\psi}$ act on a finite dimensional space, this proves the result.
\end{proof}

\begin{remark}\label{rem_contin} 
Similarly for any unitary operators $U_0,U_1 \in B(\ell^2_n)$, for any 
sequence $(F_m)_m$ of unitaries on $\ell^2_n$ such that $F_m\to U_0$ as 
$m\to\infty$, and for any $\phi\in C(\Tdb^2)$, we have
$$
T^{F_m,U_1}_{\phi}\longrightarrow
T^{U_0,U_1}_{\phi} \quad\hbox{as}\ m\to\infty.
$$
\end{remark}

\section{From Peller's problem to multiple operator integrals}

Let $f\in C^1(\Tdb)$. The divided difference of the first order
is the function $f^{[1]}\colon \Tdb^2\to\Cdb\,$
defined by
\begin{align*}
{f^{[1]} \left(z_0,z_{1} \right)} :=
\begin{cases}\frac
{f(z_0) - f(z_1)}{z_0 - z_1}, & \text{if~$z_0
\neq z_1$} \\
\frac{d}{dz}f(z)_{\vert {z=z_0}} & \text{if~$z_0=z_1$}
\end{cases}, \quad z_0, z_1\in\mathbb{T}.
\end{align*}
This is continuous function, symmetric in the two variables $(z_0,z_1)$.

Assume further that $f\in C^2(\Tdb)$. Then the divided difference of the second order
is the function $f^{[2]}\colon \Tdb^3\to\Cdb\,$
defined by
\begin{align*}
{f^{[2]} \left(z_0,z_{1},z_2 \right)} :=
\begin{cases}\frac
{f^{[1]}(z_0,z_1) - f^{[1]}(z_1,z_2)}{z_0 - z_2}, & \text{if~$z_0
\neq z_2$}, \\ \frac {d}{dz} f^{[1]}(z,z_1)_{\vert z=z_0}, & \text{if~$z_0=z_2$}
\end{cases}, \quad z_0, z_1,z_2\in\mathbb{T}.
\end{align*}
Note that $f^{[2]}$ is a continuous function, which is
symmetric in the three variables $(z_0,z_1,z_2)$.

Double and triple operator integrals built on divided differences
provide remarkable formulations for the functional calculus of normal
operators. Here we restrict to unitaries. 
First whenever $U_0,U_1\in B(\ell^2_n)$
are unitary operators and $f\in C^1(\Tdb)$, then
\begin{equation}\label{f(A)-f(B)_finite}
f(U_0)-f(U_1)=T^{U_0,U_1}_{f^{[1]}}(U_0-U_1).
\end{equation}
The elementary argument in \cite[Subsection 3.4]{CLPST1} yields this well-known identity.
See \cite[(2.4)]{Peller2005} and the references therein for the validity of that
formula in the infinite dimensional setting.

Second, let $Z\in B(\ell^2_n)$ be a self-adjoint operator
and let $U\in B(\ell^2_n)$ be a unitary operator. Then 
the function $t\mapsto f(e^{itZ}U)$ is differentiable and
\begin{equation}\label{e2}
\frac{d}{dt}\bigl(f(e^{itZ}U)\bigr)_{\vert {t=0}}
=T_{f^{[1]}}^{U,U}(iZU).
\end{equation}
Indeed by (\ref{f(A)-f(B)_finite}), we have
$$
\frac{f(e^{itZ}U) -f(U)}{t}\, = T^{e^{itZ}U,U}_{f^{[1]}}\Bigl(\frac{e^{itZ}U-U}{t}\Bigr)
$$
for any $t\not=0$. Since $\frac{d}{dt}\bigl(e^{itZ}\bigr)_{\vert t=0} = iZ$, the 
result follows from Remark \ref{rem_contin}.

The following identity may be viewed as a higher dimensional version of
\eqref{f(A)-f(B)_finite}. A similar result was established 
in \cite[Theorem 15]{CLPST1} for self-adjoint operators. The proof 
is identical so we omit it.

\begin{proposition}\label{perturbation theorem}
Let $f\in C^2(\mathbb T)$ and let $U_0,U_1,U_2\in B(\ell^2_n)$ be
unitary operators. Then for all $X\in B(\ell^2_n)$ we have
$$T^{U_0,U_2}_{f^{[1]}}(X)-T^{U_1,U_2}_{f^{[1]}}(X)=
T^{U_0,U_1,U_2}_{f^{[2]}}(U_0-U_1,X).
$$
\end{proposition}

We conclude this short section with a formula relating the second order perturbation
operator (\ref{second}) with a combination of operator integrals.

\begin{theorem}\label{th_KopliankoToTOI} 
For any self-adjoint operator $Z\in B(\ell^2_n)$, for any
unitary operator $U\in B(\ell^2_n)$ and for
any $f\in C^2(\mathbb T),$ we have
\begin{multline}\label{Taylor_MOI}
f(e^{iZ}U)-f(U)-\frac{d}{dt}\bigl(f(e^{itZ}U)\bigr)_{\vert {t=0}}\\=
T_{f^{[2]}}^{e^{iZ}U,U,U}(e^{iZ}U-U,iZU)+T_{f^{[1]}}^{e^{iZ}U,U}(e^{iZ}U-U-iZU).
\end{multline}
\end{theorem}

\begin{proof}
By \eqref{f(A)-f(B)_finite} we have
$$
f(e^{iZ}U)-f(U)=T_{f^{[1]}}^{e^{iZ}U,U}(e^{iZ}U-U).
$$
Combining with \eqref{e2}, we obtain
$$
f(e^{iZ}U)-f(U)-\frac{d}{dt}\bigl(f(e^{itZ}U)\bigr)_{\vert {t=0}} \,=\,
T_{f^{[1]}}^{e^{iZ}U,U}(e^{iZ}U-U)\, -\, T_{f^{[1]}}^{U,U}(iZU).
$$
By linearity, the right-hand side can be written as
$$
T_{f^{[1]}}^{e^{iZ}U,U}(e^{iZ}U-U -iZU)\, +\,
\bigl(T_{f^{[1]}}^{e^{iZ}U,U}(iZU)-\, T_{f^{[1]}}^{U,U}(iZU)\bigr).
$$
Applying Proposition \ref{perturbation theorem}, we obtain that
$$
T_{f^{[1]}}^{e^{iZ}U,U}(iZU)-\, T_{f^{[1]}}^{U,U}(iZU) = 
T_{f^{[2]}}^{e^{iZ}U,U,U}(e^{iZ}U -U, iZU),
$$
and this yields the desired identity (\ref{Taylor_MOI}).
\end{proof}

\section{Finite-dimensional constructions} \label{sec_FD}

In this section we establish various estimates concerning
finite dimensional operators. The symbol `${\rm const}$' will
stand for uniform positive constants, not depending on the dimension.

The estimates we are going to establish in this section start
from a result going back to \cite{AdPS}. Let $h\colon [-e^{-1}, e^{-1}]\to\Rdb\,$
be the function defined by 
$$
h(x):=
\left\{\begin{array}{rc}
|x|\Big(\log\Big|\log\frac{|x|}{e}\Big|\Big)^{-\frac12}, & x\neq 0 \\
0, & x=0
\end{array}\right..
$$
Then $h$ is a $C^1$-function. We may extend it to a 
$2\pi$-periodic $C^1$-function, that we still denote by 
$h$ for convenience.

According to \cite[Section 3]{AdPS}, there exists a constant $c>0$ and, for
any $n\geq 3$, self-adjoint operators $R_n,D_n\in B(\ell^{2}_{2n})$
such that 
\begin{equation}\label{est1}
\|R_nD_n-D_nR_n\|_\infty\le \pi 
\end{equation}
and 
\begin{equation}\label{est2}
\bignorm{R_n h(D_n)- h(D_n)R_n}_\infty \geq  c\,\log(n)^{\frac12}.
\end{equation}
By changing the dimension
from $2n$ to $2n+1$ and adding a zero on the diagonal, one may obtain the
above results for some self-adjoint
operators $R_n,D_n\in B(\ell^{2}_{2n+1})$
satisfying the additional property
\begin{equation}\label{est4}
0\in\sigma(D_n).
\end{equation}

We shall derive the following result.

\begin{theorem}\label{A and B exist}
For any  $n\ge 3$, there exist self-adjoint operators 
$A_{n},B_{n}\in B(\ell^{2}_{2n+1})$ such that $B_n\not=0$,
$0\in\sigma(A_n)$, 
$$
\bignorm{h(A_n+B_n)-h(A_n)}_\infty\ge 
{\rm const} \,\log(n)^{\frac12} \|B_n\|_\infty,
$$
and the operators
$A_{n}$ and $A_n + B_n$ are conjugate. That is, there exists a 
unitary operator $S_n\in B(\ell^{2}_{2n+1})$ such 
that $A_n+B_n = S_n^{-1}A_n S_n$.
\end{theorem}

\begin{proof} 
Let us first observe that for any $N\geq 1$ and any operators $X,Y\in B(\ell^2_N)$,
\begin{equation}\label{derivation}
\frac{e^{itX}Y - Y e^{itX}}{t}\,\longrightarrow\, i(XY-YX)\quad\hbox{as}\ t\to 0.
\end{equation}
Indeed, this follows from the fact that $\frac{d}{dt}(e^{itX})_{\vert {t=0}} = iX$.

Consider $D_n$ and $R_n$ satisfying (\ref{est1}),
(\ref{est2})
and (\ref{est4}).
For any $t>0$, define
$$
B_{n,t}:=e^{itR_n}D_n e^{-itR_n}- D_n.
$$

On the one hand, applying (\ref{derivation}) with $X=R_n$ and $Y=D_n$,
we obtain that
\begin{align*}
\frac1t \|B_{n,t}\|_\infty &=
\frac1t \bignorm{e^{itR_n}D_ne^{-itR_n}- D_n}_\infty \\ &
=\frac1t \bignorm{e^{itR_n}D_n- D_n e^{itR_n}}_\infty\\ &
\longrightarrow \|R_nD_n-D_nR_n\|_\infty
\end{align*}
as $t\to 0$.

On the other hand, using the identity 
$$
h(e^{itR_n}D_ne^{-itR_n}) = 
e^{itR_n}h(D_n)e^{-itR_n}
$$ 
and applying \eqref{derivation} 
with $X=R_n$ and $Y=h(D_n)$, we have
\begin{align*}
\frac1t\bignorm{h(D_n+B_{n,t})- h(D_n)}_\infty 
& =
\frac1t\bignorm{e^{itR_n}h(D_n)e^{-itR_n}-h(D_n)}_\infty
\\ & =
\frac1t\bignorm{e^{itR_n}h(D_n)-h(D_n)e^{itR_n}}_\infty
\\ & \longrightarrow
\|R_n h(D_n)-h(D_n)R_n\|_\infty
\end{align*}
as $t\to 0$.

Therefore, there exists $t>0$ such that
\begin{equation}\label{minor}
\frac{t}{2}\,\|R_nD_n-D_nR_n\|_\infty\leq \|B_{n,t}\|_\infty \le 2\pi t  
\end{equation}
and
$$
\bignorm{h(D_n+B_{n,t})-h(D_n)}_\infty\ge c\,\frac{\log(n)^{\frac12}}{2} \, t.
$$
The above two estimates lead to
$$
\bignorm{h(D_n+B_{n,t})-h(D_n)}_\infty\ge \frac{c}{4\pi}\, \log(n)^{\frac12}\,\norm{B_{n,t}}_\infty.
$$
Furthermore property (\ref{est2}) implies that $D_n$ and $R_n$ do not commute. 
Hence the first inequality in (\ref{minor}) ensures that $B_{n,t}\not =0$. 

To get the result, we set $A_n=D_n$ and $B_n=B_{n,t}$. According to the definition of 
$B_{n,t}$, the operators $A_n$ and $A_n+B_n$ are conjugate. All other properties of
the statement of the theorem follow from the above estimates and 
(\ref{est4}).
\end{proof}

Let $g\in C^1(\mathbb T)$ be the unique function satisfying
\begin{equation}\label{def_g_0}
g(e^{i\theta})=h(\theta), \qquad \theta\in \Rdb.
\end{equation}
The following theorem translates the preceding result into the setting of unitary operators.

\begin{theorem}\label{U and V exist}
For any $n\ge 3$, 
there exist unitary operators $H_{n},K_{n}\in B(\ell^{2}_{2n+1})$  such that
$$
H_n\not= K_n,\quad
\sigma(H_n)=\sigma(K_n),\quad 1\in\sigma(H_n),
$$
and
\begin{equation}\label{U and V exist_3}
\|g(K_n)-g(H_n)\|_\infty\ge {\rm const}\, \log(n)^\frac12 \|K_n-H_n\|_\infty.
\end{equation}
\end{theorem}

\begin{proof} Given any $n\ge 3$, let $A_n,B_n$ be the operators 
from Theorem \ref{A and B exist}, and set
$$
H_{n}=e^{i A_n}\qquad\hbox{and}\qquad K_{n}=e^{i(A_n+B_n)}.
$$
These are unitary operators. Since $A_n$ and $A_n+ B_n$ are conjugate, they have the 
same spectrum hence in turn, $\sigma(H_{n})=\sigma(K_{n})$. Moreover $1\in\sigma(H_{n})$
since $0\in\sigma(A_n)$. Since $A_n$ and $A_n+B_n$ are conjugate but different,
their sets of spectral projections are different.
This implies that $H_n\not= K_n$.

By construction we have
$$
g(H_{n})=h(A_n)\qquad\hbox{and}\qquad
g(K_{n})=h(A_n+B_n).
$$
Therefore, by Theorem \ref{A and B exist}, we have
$$
\norm{g(K_{n}) - g(H_{n})}_\infty\geq {\rm const} \log(n)^{\frac12}\norm{B_n}_\infty.
$$
Moreover 
$$
\|K_n-H_n\|_\infty =\bignorm{e^{i(A_n+B_n)} -e^{iA_n}}_\infty\leq\norm{B_n}_\infty
$$
by \cite[Lemma 8]{PS-Lipschitz}. This yields the result.
\end{proof}

Let $f\colon\Tdb\to \Cdb\,$ be defined by
\begin{equation}\label{def_g_1}
f(z) =(z-1)g(z),\qquad z\in\Tdb.
\end{equation}
It turns out that $f\in C^2(\mathbb T)$. This 
follows from the definition
of $h$, which is $C^2$ on $(-e^{-1},e^{-1})\setminus\{0\}$, 
and the fact that $\lim_{x\to 0} xh''(x)=0$.
Details are left to the reader. 

We also define an auxiliary function
$\varsigma:\mathbb T^3\to \mathbb C$ given by
\begin{equation}\label{def_varsigma}
\varsigma(z_0,z_1,z_2)=z_1 f^{[2]}(z_0,z_1,z_2).
\end{equation}

\begin{lemma}\label{one is zero}
For any $z_0,z_2\in\Tdb$, we have 
$$
\varsigma(z_0,1,z_2)=g^{[1]}(z_0,z_2).
$$
\end{lemma}

\begin{proof}
By the definition of $\varsigma,$ and since $z_1=1$, it is enough to prove that
$$
f^{[2]}(z_0,1,z_2)=g^{[1]}(z_0,z_2), \qquad z_0,z_2\in\Tdb.
$$

We have to consider several different cases.
Let us first assume that $z_0\neq z_2.$ If $z_0\neq 1$ and $z_2\neq 1,$ then we have
\begin{align*}
f^{[2]}(z_0,1,z_2)& =\frac{f^{[1]}(z_0,1)-f^{[1]}(1,z_2)}{z_0-z_2}
=\frac{\frac{f(z_0)-f(1)}{z_0-1}-\frac{f(1)-f(z_2)}{1-z_2}}{z_0-z_2}\\
& =\frac{g(z_0)-g(z_2)}{z_0-z_2}=g^{[1]}(z_0,z_2).
\end{align*}
If $z_0=1$ and $z_2\neq 1,$ then using $\frac{d}{dz}f(z)_{\vert z=1} =  g(1) =h(0)
=0$, we have
\begin{align*}
f^{[2]}(1,1,z_2)& =\frac{f^{[1]}(1,1)-f^{[1]}(1,z_2)}{1-z_2}=
\frac{\frac{d}{dz}f(z)_{\vert z=1}-\frac{f(1)-f(z_2)}{1-z_2}}{1-z_2}\\ &
=\frac{-g(z_2)}{1-z_2}=g^{[1]}(1,z_2).
\end{align*}
The argument is similar, when $z_0\neq 1$ and $z_2= 1.$

Assume now that $z_0=z_2.$
In this case, we obtain that
\begin{align*}
f^{[2]}(z_0,1,z_0) & =\frac{d}{dz}f^{[1]}(z,1)_{\vert z=z_0}=
\frac{d}{dz}\Big(\frac{f(z)-f(1)}{z-1}\Big)_{\vert z=z_0}\\ &= 
\frac{d}{dz}g(z)_{\vert z=z_0}=g^{[1]}(z_0,z_0).
\end{align*}
\end{proof}

\begin{corollary}\label{MOI_estimate}
For any $n\ge 3$, there exist unitary operators 
$H_{n},K_{n}\in B(\ell^{2}_{2n+1})$  such that 
\begin{equation*}
\sigma(H_n)=\sigma(K_n),
\end{equation*} 
and
\begin{equation}\label{221}
\big\|T_{\varsigma}^{K_n,H_{n},H_{n}}\colon\mathcal S^2_{2n+1}\times
\mathcal S^2_{2n+1}\to \mathcal S^1_{2n+1} \big\|\ge {\rm const}\,\log(n)^{\frac12}.
\end{equation}
\end{corollary}

\begin{proof}
Take $H_n, K_n$ as in Theorem \ref{U and V exist}; 
these unitary operators have the same spectrum.
Let $\{\mu_k\}_{k=1}^{2n+1}$ be the sequence of eigenvalues of the operator $H_n$,
counted with multiplicity. Since $1\in\sigma(H_n)$, we may assume that 
$\mu_1=1$. According to Lemma \ref{corres} and Theorem \ref{key}, we have
\begin{equation*}
\big\|T_{\varsigma}^{K_n,H_n,H_n}\colon\mathcal S^2_{2n+1}\times
\mathcal S^2_{2n+1}\to \mathcal S^1_{2n+1} \big\|\,=\,\max_{1\le k\le {2n+1}}
\big\|T_{\varsigma_k}^{K_n,H_n}\colon\mathcal S^\infty_{2n+1}\to \mathcal S^\infty_{2n+1}\big\|,
\end{equation*}
where, for any $k=1,\ldots, 2n+1$, we set
$$
\varsigma_k(z_0,z_1):=\varsigma(z_0,\mu_k, z_1), \ \ z_0,z_1\in\mathbb T.
$$
In particular, the inequality
\begin{equation*}
\big\|T_{\varsigma}^{K_n,H_n,H_n}\colon\mathcal S^2_{2n+1}\times
\mathcal S^2_{2n+1}\to \mathcal S^1_{2n+1} \big\|\ge
\big\|T_{\varsigma_1}^{K_n,H_n}\colon\mathcal S^\infty_{2n+1}\to \mathcal S^\infty_{2n+1}\big\|
\end{equation*}
holds. From Lemma \ref{one is zero}, we have that
$$
\varsigma_1(z_0,z_1)=\varsigma(z_0, 1, z_1)=g^{[1]}(z_0,z_1).
$$
Therefore, we obtain 
\begin{equation}\label{connection between DOI and MOI_inequality}
\big\|T_{\varsigma}^{K_n,H_n,H_n}\colon\mathcal S^2_{2n+1}\times
\mathcal S^2_{2n+1}\to \mathcal S^1_{2n+1} \big\|\ge  
\big\|T_{g^{[1]}}^{K_n,H_n}\colon\mathcal S^\infty_{2n+1}\to \mathcal S^\infty_{2n+1}\big\|.
\end{equation}
Since $H_n\not= K_n$, we derive
$$
\big\|T_{\varsigma}^{K_n,H_n,H_n}\colon\mathcal S^2_{2n+1}\times
\mathcal S^2_{2n+1}\to \mathcal S^1_{2n+1} \big\|\ge  
\,\frac{\bignorm{T_{g^{[1]}}^{K_n,H_n}(K_n-H_n)}_\infty}{\norm{K_n-H_n}_\infty}\,.
$$
From the identity \eqref{f(A)-f(B)_finite}, we have 
$T_{g^{[1]}}^{K_n,H_n}(K_n-H_n)= g(K_n) -g(H_n)$. Hence the 
above inequality means that
$$
\big\|T_{\varsigma}^{K_n,H_n,H_n}\colon\mathcal S^2_{2n+1}\times
\mathcal S^2_{2n+1}\to \mathcal S^1_{2n+1} \big\|\ge  
\,\frac{\norm{g(K_n) -g(H_n)}_\infty}{\norm{K_n-H_n}_\infty}\,.
$$
Applying (\ref{U and V exist_3}) we obtain the desired estimate.
\end{proof}

We are now ready to prove the final estimate of this section.

\begin{corollary}\label{thm_mainTOIestimate}
For any $n\ge 3$, there exist a self-adjoint operator $W_n\in B(\ell^{2}_{8n+4})$
with $\|W_n\|_2\le 1$ and a unitary operator $U_n\in B(\ell^{2}_{8n+4})$ such that
\begin{equation}\label{mainTOIestimate}
\Big\|T_{f^{[2]}}^{U_n,U_n,U_n}(W_n U_n,W_n U_n)\Big\|_1\ge 
{\rm const} \,\log(n)^{\frac12}.
\end{equation}
\end{corollary}

\begin{proof}
We take $H_n$ and $K_n$ given by Corollary \ref{MOI_estimate}. 
Then we consider 
\begin{equation}\label{def_H_n} V_n:=\left(\begin{array}{cc}
K_n & 0 \\ 0 & H_n
\end{array}\right)\qquad \hbox{and then}\qquad
U_n:=\left(\begin{array}{cc}
V_n & 0 \\ 0 &  V_n
\end{array}\right).
\end{equation}
Then $V_n$ is a unitary operator acting on $\ell^{2}_{4n+2}$
and $U_n$ is a unitary operator acting on $\ell^{2}_{8n+4}$.

We claim that there exists a self-adjoint operator $W_n\in B(\ell^{2}_{8n+4})$
such that $\|W_n\|_2\le 1$ and
$$
\bignorm{T_{\varsigma}^{U_n,U_n,U_n}(W_n,W_n)}_1\geq{\rm const}\,\log(n)^{\frac12}.
$$
Indeed, using (\ref{221}) and the fact that $H_n$ and $K_n$ have the same sprectrum,
this follows from the proofs of \cite[Lemmas 22-25]{CLPST1}. Indeed the arguments
there can be used word for word in the present case. It therefore suffices to show
\begin{equation}\label{goal}
\bignorm{T_{\varsigma}^{U_n,U_n,U_n}(W_n,W_n)}_1
= \bignorm{T_{f^{[2]}}^{U_n,U_n,U_n}(W_n U_n, W_n U_n)}_1.
\end{equation}

For that purpose we set $N=8n+4$ and consider a spectral decomposition 
$U_n = \sum_{i=1}^N z_i  P_i\,$ of $U_n$.
Then by definition (\ref{bilinearMOI}) we have 
\begin{align*}
T_{f^{[2]}}^{U_n,U_n,U_n}(W_n U_n,W_n U_n) & =
\sum_{i,j,k=1}^{N} f^{[2]}(z_i,
z_k,z_j)P_i (W_nU_n) P_k
(W_nU_n) P_j \\
& =\sum_{i,j,k=1}^{N} f^{[2]}(z_i,
z_k,z_j) P_i W_n\Bigl(\sum_{l=1}^{N}z_l P_l\Bigr)P_k
W_n P_j U_n\\
& =\sum_{i,j,k=1}^{N} z_k f^{[2]}(z_i, z_k,z_j) P_i W_n P_k
W_n P_j  U_n\\ 
& \stackrel{\eqref{def_varsigma}}{=}\sum_{i,j,k=1}^{N} \varsigma(z_i,
z_k,z_j) P_i W_n P_k
W_n P_j U_n\\
& =T_{\varsigma}^{U_n,U_n,U_n}(W_n,W_n) U_n.
\end{align*}
Since $U_n$ is a unitary, this equality implies (\ref{goal}), which
completes the proof.
\end{proof}

\section{A solution to Peller's problem for unitary operators}

In this section, we answer Peller's question
raised in \cite[Problem 1]{Peller2005} in the negative.

\begin{theorem}\label{main} 
There exist a function $f\in C^2(\Tdb)$, a separable Hilbert
space $\mathcal H$, a unitary operator $U\in B(\mathcal H)$ and a 
self-adjoint operator $Z\in \mathcal S^2(\mathcal H)$ such that
\begin{equation}\label{result}
f\bigl(e^{iZ} U\bigr) - f(U) -\,
\frac{d}{dt}\bigl(f(e^{itZ}U)\bigr)_{\vert t=0}\,\notin\, \mathcal S^1(\mathcal H).
\end{equation}
\end{theorem}

In the above statement, $\frac{d}{dt}\bigl(f(e^{itZ}U)\bigr)_{\vert t=0}$ denotes the derivative
of this function at $t=0$. We refer to \cite[(2.7)]{Peller2005} and the references therein for the facts that
for any $f\in C^1(\Tdb)$, for any unitary operator $U\in B(\mathcal H)$ and any
self-adjoint operator $Z\in \mathcal S^2(\mathcal H)$, the difference operator
$f\bigl(e^{iZ} U\bigr) - f(U)$ belongs to $\mathcal S^2(\mathcal H)$ and the function
$t\mapsto f(e^{itZ}U)$ is differentiable from $\Rdb$ into $\mathcal S^2(\mathcal H)$.
Therefore, the operator in \eqref{result} belongs to $S^2(\mathcal H)$.

Theorem \ref{main} will be proved with the function $f$ given by (\ref{def_g_1}).
We will combine a direct sum argument and the following lemma, whose proof relies 
on Corollary \ref{thm_mainTOIestimate}.

\begin{lemma}\label{lem_sup} 
For any $n\geq 1$, there exist a non zero
self-adjoint operator $Z_n\in B(\ell^{2}_{8n+4})$ 
and a unitary operator $U_n\in B(\ell^{2}_{8n+4})$, 
such that
\begin{equation}\label{lem1}
\sum_{n=1}^\infty \| Z_n\|_2^2\,<\infty,
\end{equation}
and
\begin{equation}\label{lem2}
\lim_{n\to\infty}\,
\frac{\ \Big\| f(e^{i Z_n} U_n)-f(U_n)-\frac{d}{dt}
\bigl(f(e^{it Z_n}U_n)\bigr)_{\vert t=0}\Big\|_1\,}
{\|Z_n\|_2^2}\,=\,\infty.
\end{equation}
\end{lemma}

\begin{proof}
We fix $n\geq 3$ and we take $W_n$ and $U_n$ given by 
Corollary \ref{thm_mainTOIestimate}.
Note that changing $W_n$ into $\norm{W_n}_2^{-1}W_n$, we may (and do) 
assume that $\norm{W_n}_2=1$.
We consider the sequence
$$
W_{m,n}=\frac{1}m W_n, \quad m\ge 1,
$$
and we set
$$
R_{m,n}:=
f(e^{iW_{m,n}}U_n)-f(U_n)-\frac{d}{dt}\bigl(f(e^{it W_{m,n}}U_n)\bigr)_{\vert t=0}.
$$
By Theorem \ref{th_KopliankoToTOI} we have
\begin{multline}\label{e3} m^2 R_{m,n}=T_{f^{[2]}}^{e^{iW_{m,n}}
U_n,U_n,U_n}\bigl(m(e^{iW_{m,n}}U_n-U_n),i W_{n}U_n\bigr)
\\ + T_{f^{[1]}}^{e^{iW_{m,n}} U_n, U_n}
\bigl(m^2(e^{iW_{m,n}}U_n-U_n-iW_{m,n}U_n)\bigr).
\end{multline}
Note that 
$$
m\bigl(e^{iW_{m,n}}-I_n\bigr)\longrightarrow i W_{n} \quad\text{as} \ m\to\infty.
$$
Hence by Lemma \ref{lem_contin}, we have
$$
T_{f^{[2]}}^{e^{iW_{m,n}}U_n,U_n,U_n}\bigl(m(e^{iW_{m,n}}U_n-U_n),i W_{n}U_n\bigr)
\longrightarrow
T_{f^{[2]}}^{U_n,U_n,U_n}(iW_{n}U_n,i W_{n} U_n)
$$
as $m\to \infty.$
This result and Corollary \ref{thm_mainTOIestimate} imply that for $m$ large enough, we have
\begin{equation}\label{eq1}
\Big\|T_{f^{[2]}}^{e^{iW_{m,n}}U_n,U_n,U_n}\bigl(m(e^{iW_{m,n}}U_n-U_n),iW_{n}U_n\bigr)
\Big\|_1\ge {\rm const}\, \log(n)^{\frac12}.
\end{equation}

We now turn to the analysis of the second term in the right hand side of (\ref{e3}).
Since $f\in C^2(\mathbb T)$, there exists a constant $K>0$ 
(only depending on $f$ and not on either $n$ or the 
operators $U_n$ and $W_{m,n}$) such that
$$
\bignorm{T_{f^{[1]}}^{e^{iW_{m,n}}U_n,U_n}\colon \mathcal 
S^1_{8n+4}\to \mathcal S^1_{8n+4}}\,\leq K.
$$
This follows from \cite{BiSo3} (see also \cite{Peller1985}).

Now observe that
$$
m^2\bigl(e^{iW_{m,n}}-I_n-iW_{m,n}\bigr)
\longrightarrow \frac{W_n^2}{2}  \quad\text{as} \ m\to\infty.
$$
Hence we have
\begin{equation}\label{eq2}
\Bignorm{T_{f^{[1]}}^{e^{iW_{m,n}}U_n,U_n}\bigl(m^2(e^{iW_{m,n}}
U_n -U_n -iW_{m,n}U_n)\bigr)}_1 \leq K \norm{W_n^2}_1
= K\norm{W_n}_2^2
\end{equation}
for $m$ large enough.

Combining (\ref{eq1}) and (\ref{eq2}), we deduce from the identity (\ref{e3})
the existence of an integer $m\geq 1$ 
for which we have an estimate 
\begin{equation}\label{R}
m^2\norm{R_{m,n}}_1\geq{\rm const}\,\log(n)^{\frac12}.
\end{equation}
We may assume that $m\geq n$, which ensures that 
$$
\norm{W_{m,n}}_2\leq\frac1n.
$$
Then we set $Z_n = W_{m,n}$. The preceding inequality
implies that $\sum_n\norm{Z_n}_2^2<\infty$.
Since $\norm{W_n}_2=1$, we have $\norm{Z_{n}}_2=\frac1m$ hence
the estimate (\ref{R}) yields (\ref{lem2}). 
\end{proof}

\begin{proof}[Proof of Theorem \ref{main}]
We apply Lemma \ref{lem_sup} above. We set
$$
\beta_n:=\bigl\|
f(e^{i Z_n}U_n)-f(U_n)-\frac{d}{dt}
\bigl(f(e^{it  Z_n}U_n)\bigr)_{\vert t=0}\bigr\|_1
$$
for any $n\geq 1$. Since $\bigl\{\beta_n\norm{Z_n}^{-2}_2\bigr\}_{n=1}^{\infty}$ 
is an unbounded sequence, by (\ref{lem2}), there exists
a positive sequence $(\alpha_n)_{n\geq 1}$ such that 
\begin{equation}\label{alpha}
\sum_{n=1}^\infty 
\alpha_n<\infty\qquad\hbox{and}\qquad
\sum_{n=1}^\infty 
\alpha_n\beta_n\norm{Z_n}^{-2}_2 = \infty.
\end{equation}
Set 
$$
N_n = \bigl[\alpha_n \norm{Z_n}^{-2}_2\bigr] +1, 
$$
where
$[\,\cdotp]$ denotes the integer part of a real number.
We have both
$$
N_n\norm{Z_n}^2_2\leq \alpha_n+\norm{Z_n}^2_2
\qquad\hbox{and}\qquad
N_n\geq \alpha_n\norm{Z_n}^{-2}_2.
$$
Hence 
it follows from (\ref{alpha}), (\ref{lem1}) and (\ref{lem2}) that
$$
\sum_{n=1}^\infty N_n\norm{Z_n}^2_2\,<\infty
\qquad\hbox{and}\qquad
\sum_{n=1}^\infty N_n\beta_n\,=\infty.
$$

We let $\mathcal H_n = \ell^2_{N_n}(\ell^2_{8n+4})$ and we let 
$\tilde Z_n$ (resp. $\tilde U_n$) be the element of
$B(\mathcal H_n)$ obtained as the direct sum of $N_n$ copies
of $Z_n$ (resp. $U_n$). Then $\tilde Z_n$ is a self-adjoint operator
and $\norm{\tilde Z_n}^2_2 = N_n\norm{Z_n}^2_2$. Consequently,
\begin{equation}\label{S2}
\sum_{n=1}^\infty \norm{\tilde Z_n}^2_2\,<\infty\,.
\end{equation}
Likewise $\tilde U_n$ is a unitary operator and we have
\begin{align*}
\bigl\|
f(e^{i \tilde Z_n}\tilde U_n)& -f(\tilde U_n)-\frac{d}{dt}
\bigl(f(e^{it  \tilde Z_n}\tilde U_n)\bigr)_{\vert t=0}\bigr\|_1\\
& = N_n \bigl\|
f(e^{i Z_n}U_n)-f(U_n)-\frac{d}{dt}
\bigl(f(e^{it  Z_n}U_n)\bigr)_{\vert t=0}\bigr\|_1 = N_n\beta_n.
\end{align*}
Hence
$$
\sum_{n=1}^\infty\,
\bigl\|
f(e^{i \tilde Z_n}\tilde U_n) -f(\tilde U_n)-\frac{d}{dt}
\bigl(f(e^{it  \tilde Z_n}\tilde U_n)\bigr)_{\vert t=0}\bigr\|_1\,=\infty.
$$

We finally consider the direct sum
$$
\mathcal H = \overset{2}{\oplus}_{n\geq 1} \mathcal H_n.
$$
We let $Z$ be the direct sum of the $\tilde Z_n$, defined by
$Z(\xi) = \{\tilde Z_n(\xi_n)\}_{n=1}^\infty$ for any $\xi= \{\xi_n\}_{n=1}^\infty$
in $\mathcal H$. Property (\ref{S2}) ensures that $Z$ is well-defined and belongs to
$\mathcal S^2(\mathcal H)$, with $\norm{Z}^2_2 = \sum_{n=1}^\infty \norm{\tilde Z_n}^2_2$.
Likewise we let $U$ be the direct sum of the $\tilde U_n$. This is a unitary operator
and $\frac{d}{dt}\bigl(f(e^{it  Z}U)\bigr)_{\vert t=0}$ is the direct sum of the
$\frac{d}{dt} \bigl(f(e^{it  \tilde Z_n}\tilde U_n)\bigr)_{\vert t=0}$. Therefore
\begin{align*}
\bigl\|
f(e^{i Z} U) & -f(U)-\frac{d}{dt}
\bigl(f(e^{it  Z}U)\bigr)_{\vert t=0}\bigr\|_1\\&  =
\sum_{n=1}^\infty\, 
\bigl\|
f(e^{i \tilde Z_n}\tilde U_n) -f(\tilde U_n)-\frac{d}{dt}
\bigl(f(e^{it  \tilde Z_n}\tilde U_n)\bigr)_{\vert t=0}\bigr\|_1\,.
\end{align*}
Since this sum is infinite, we obtain the assertion (\ref{result}).
\end{proof}

\bigskip\noindent
\textbf{Acknowledgement}. The authors C.C and C.L have been 
supported by the research program ANR 2011 BS01 008 01
and by the ``Conseil r\'egional de Franche-Comt\'e".
The authors D.P., F.S and A.T. have been supported by the 
Australian Research Council grant DP150100920.

\bigskip


\begin{thebibliography}{10}

\bibitem{AdPS} W. van Ackooij, B. de Pagter, F. A. Sukochev,  
\textit{Domains of infinitesimal generators of automorphism flows}, 
J. Funct. Anal. \textbf{218} (2005), no. 2, 409–424.

\bibitem{ACS} P. J. Ayre, M. G. Cowling, F. A. Sukochev, \textit{ Operator Lipschitz estimates in the unitary setting,} Proc. Amer. Math. Soc. (to appear).



%

%





\bibitem{BiSo3} M. S. Birman and M. Z. Solomyak, Double Stieltjes operator
integrals III (in Russian), Probl. Math. Phys., Leningrad Univ., 6
(1973), 27--53.

\bibitem{BY} M. Sh. Birman, D. R. Yafaev,  \textit{The spectral shift function.
The papers of M. G. Kre{\u\i}n and their further development} (Russian). Algebra i
Analiz \textbf{4} (1992), no. 5, 1--44. Translation in St. Petersburg Math.
J. \textbf{4} (1993), no. 5, 833--870.

\bibitem{CPSZ1} M. Caspers, D. Potapov, F. Sukochev, D. Zanin, {\it Weak type estimates for the absolute value mapping.} J. Operator Theory, {\bf 73} (2015), no. 2, 101--124.

\bibitem{CPSZ2} M. Caspers, D. Potapov, F. Sukochev, D. Zanin, {\it Weak type commutator and Lipschitz estimates: resolution of the Nazarov-Peller conjecture,} arXiv:1506.00778.


\bibitem{CLPST1} C. Coine, C. Le Merdy, D. Potapov, F. Sukochev, A. Tomskova, 
{\it Resolution of Peller's problem concerning Koplienko-Neidhardt trace formulae,}  
arXiv:1504.03843.


\bibitem{Farforovskaya1972} 
Yu. B. Farforovskaya, \emph{An example of a {L}ipschitz function of self-adjoint operators
with non-nuclear difference under a nuclear perturbation.}, Zap. Nauchn. Sem.
Leningrad. Otdel. Mat. Inst. Steklov. (LOMI) \textbf{30} (1972), 146--153.

\bibitem{Krein1} M. G. Kre{\u\i}n, {\it On the trace formula in perturbation theory} 
(Russian), Mat. Sbornik N.S. {\bf 33} (75), (1953), 597--626.

\bibitem{Krein} M. G. Kre{\u\i}n, {\it On the perturbation determinant and the trace 
formula for unitary and self-adjoint operators} (Russian), Dokl. Akad. Nauk SSSR {\bf 144} (1962), 268--271. 
Translation: Soviet Math. Dokl. {\bf 3} (1962), 707--710.

\bibitem{KreinPerturbation1964}
M.~G. Kre{\u\i}n, \emph{Some new studies in the theory of perturbations of
self-adjoint operators}, First {M}ath. {S}ummer {S}chool, {P}art {I}
({R}ussian), Izdat. ``Naukova Dumka'', Kiev, 1964, pp.~103--187.


















\bibitem{Peller1985} V. V. Peller, {\it Hankel operators in the theory of perturbations of unitary 
and selfadjoint operators}, Funktsional. Anal. i Prilozhen. 19 (1985), no. 2, 37--51 (in Russian),
English translation: Funct. Anal. Appl. 19 (1985), 111--123.


\bibitem{Peller1990} V. V. Peller, {\it  Hankel operators in the perturbation theory of unbounded 
selfadjoint operators}. Analysis and partial differential equations, 529--544, Lecture Notes in 
Pure and Appl. Math., 122, Dekker, New York, 1990. 


\bibitem{Peller2005} V. V. Peller, {\it An extension of the Koplienko-Neidhardt trace formulae.} 
J. Funct. Anal. {\bf 221} (2005), no. 2, 456--481.











\bibitem{PS-Lipschitz}
D.~Potapov and F.~Sukochev, \emph{Operator-{L}ipschitz functions
in {S}chatten-von {N}eumann classes},  Acta Math. \textbf{207}
(2011), no. 2, 375--389.



\end{thebibliography}
\end{document}